\documentclass{{amsart}}
\usepackage{amsmath, amssymb, amsthm}
\usepackage{enumerate}

\newtheorem{lemma}{Lemma}[section]
\newtheorem{proposition}{Proposition}[section]
\newtheorem*{theorema}{Theorem A}
\newtheorem*{theoremb}{Theorem B}
\newcommand{\taf}{{\hskip 5pt} $\blacksquare$
                  \renewcommand{\qedsymbol}{}}
\begin{document}
\title{From dyadic  $\Lambda_{\alpha}$ to $\Lambda_{\alpha}$}

\author{Wael Abu-Shammala}
\address{Department of Mathematics\\ Indiana University\\
Bloomington IN 47405} \email{wabusham@indiana.edu}
\author{Alberto Torchinsky}
\address{Department of Mathematics\\ Indiana University\\
Bloomington IN 47405} \email{torchins@indiana.edu}
\begin{abstract} In this paper we show how to compute the
$\Lambda_{\alpha}$ norm , $\alpha\ge 0$, using the dyadic grid. This
result is a consequence of the description of the Hardy spaces
$H^p(R^N)$ in terms of dyadic and special atoms.
\end{abstract}
\maketitle

Recently, several novel methods for computing  the BMO norm of a
function $f$ in two dimensions  were discussed in \cite{LV}. Given
its importance, it is  also of interest to explore the possibility
of computing the norm of a BMO function, or more generally  a
function in the Lipschitz class $\Lambda_{\alpha}$, using the
dyadic grid in $R^N$. It turns out that the BMO question is
closely related to that of approximating functions in the Hardy
space $H^1(R^N)$ by the Haar system. The approximation in
$H^1(R^N)$ by affine systems was proved in \cite{HB}, but this
result does not apply to the Haar system. Now, if $H^A(R)$ denotes
the closure of the Haar system in $H^1(R)$, it is not hard to see
that the distance $d(f,H^A)$ of $f\in H^1(R)$ to $H^A$ is $\sim
\big|\int_0^{\infty}f(x)\,dx\big|$, see \cite{AST}. Thus, neither
dyadic atoms suffice to describe the Hardy spaces, nor the
evaluation of the norm in BMO can be reduced to a straightforward
computation using the dyadic intervals. In this paper we address
both of these issues. First, we give a characterization of the
Hardy spaces $H^p(R^N)$ in terms of dyadic and special atoms, and
then, by a duality argument, we show how to compute the norm in
$\Lambda_{\alpha}(R^N)$, $\alpha\ge 0$, using the dyadic grid.

We begin by introducing some notations.  Let ${\mathcal J}$ denote a
family of cubes $Q$ in $R^N$, and ${\mathcal P}_d$ the collection of
polynomials in $R^N$ of degree less than or equal to $d$. Given
$\alpha\ge 0$, $Q\in {\mathcal J}$, and a locally integrable
function $g$, let $p_Q(g)$ denote the unique polynomial in
${\mathcal P}_{[\alpha]}$ such that $[g-p_Q(g)]\,\chi_Q$ has
vanishing moments up to order $[\alpha]$.

For a locally square-integrable function  $g$, we consider the
maximal function $M^{\sharp, 2}_{\alpha,{\mathcal J}}g(x)$ given
by
\[M^{\sharp, 2}_{\alpha,{\mathcal J}}g(x)=\sup_{x\in Q,\, Q\in {\mathcal J}}
 \frac{1}{|Q|^{\alpha/N}}\bigg(\frac{1}{|Q|} \int_Q
 |g(y)-p_Q(g)(y)|^2\,dy
 \bigg)^{1/2}\,.
 \]

The Lipschitz space $\Lambda_{\alpha, {\mathcal J}}$ consists of
those  functions $g$ such that
 $M^{\sharp,2}_{\alpha, {\mathcal J}}g$ is in $L^{\infty}$,
 $\|g\|_{\Lambda_{\alpha, {\mathcal J}}}=
 \|M^{\sharp,2}_{\alpha, {\mathcal J}}g\|_{\infty}$;
  when the family in question contains all cubes in $R^N$, we
simply omit the subscript ${\mathcal J}$. Of course, $\Lambda_0 =
{\rm BMO}$.

Two other families, of dyadic nature, are of interest to us.
Intervals in $R$ of the form $I_{n,k}=[\, (k-1) 2^n, k2^n]$, where
$k$ and $n$ are arbitrary integers, positive, negative or $0$, are
said to be dyadic. In $R^N$, cubes which are the product of dyadic
intervals of the same length, i.e., of the form
$Q_{n,k}=I_{n,k_1}\times\cdots\times I_{n,k_N}$, are called dyadic,
and the collection of all such cubes is denoted ${\mathcal D}$.

There is also the family ${\mathcal D}_0$. Let $I_{n,k}'=[(k-1)2^n,
(k+1)2^n]$, where $k$ and $n$ are arbitrary integers. Clearly
$I_{n,k}'$ is dyadic
 if $k$ is odd, but not if $k$ is even. Now, the collection
 $\{I_{n,k}':n,k\ {\rm integers}\}$ contains all dyadic intervals
 as well as the shifts $[(k-1)2^n+2^{n-1}, k\,2^n+2^{n-1}]$ of the dyadic
 intervals by their half length.
In $R^N$,  put ${\mathcal
D}_0=\{Q_{n,k}':Q_{n,k}'=I_{n,k_1}'\times\cdots\times
 I_{n,k_N}'\}$;  $Q_{n,k}'$ is
called a special cube. Note that
  ${\mathcal D}_0$ contains ${\mathcal  D}$ properly.

Finally, given $I_{n,k}'$, let $I^{\,'L}_{n,k}=[(k-1)2^n,k2^n]$,
and $I^{\,'R}_{n,k}=[k2^n,(k+1)2^n]$. The $2^N$ subcubes of
$Q_{n,k}'=I_{n,k_1}'\times\cdots\times
 I_{n,k_N}'$ of the form $I_{n,k_1}^{\,'S_1}\times\cdots
 \times I_{n,k_N}^{\,'S_N}$, $S_j=L$ or $R$, $1\le j\le N$,
are called the dyadic subcubes of $Q_{n,k}'$.

Let $Q_0$ denote the special cube $[-1,1]^N$. Given $\alpha\ge 0$,
we construct a family ${\mathcal S_{\alpha}}$ of piecewise
polynomial splines in $L^2(Q_0)$ that will be useful in
characterizing $\Lambda_{\alpha}$. Let $A$ be the subspace of
$L^2(Q_0)$ consisting of all functions with vanishing moments up to
order $[\alpha]$ which coincide with a polynomial in ${\mathcal
P}_{[\alpha]}$  on each of the $2^{N}$ dyadic subcubes of $Q_0$. $A$
is a finite dimensional subspace of $L^2(Q_0)$, and, therefore, by
the Graham-Schmidt orthogonalization process, say, $A$ has an
orthonormal basis in $L^2(Q_0)$ consisting of functions
$p^1,\ldots,p^M$ with vanishing moments up to order $[\alpha]$,
which coincide with a polynomial in ${\mathcal P}_{[\alpha]}$  on
each dyadic subinterval of $Q_0$. Together with each $p^L$ we also
consider all dyadic dilations and integer translations given by
\[p_{n,k,\alpha}^L(x)=2^{n(N+\alpha)}p^L(2^nx_1+k_1,\dots,2^nx_N+k_N)\,,
\quad 1\le L\le M\,,\] and let \[{\mathcal
S_{\alpha}}=\{p_{n,k,\alpha}^L : n,k\ {\rm integers},\, 1\le L\le
M\}\,.\]

Our first result shows how the dyadic grid can be used to compute
the norm in $\Lambda_{\alpha}$.
\begin{theorema} Let $g$ be  a
locally square-integrable function and $\alpha \ge 0$. Then, $g\in
\Lambda_{\alpha}$ if, and only if, $g\in \Lambda_{\alpha,{\mathcal
D}}$ and $A_{\alpha}(g)=\sup_{p\in {\mathcal
S}_{\alpha}}\big|\langle g, p \rangle\big|<\infty$. Moreover,
\[\|g\|_{\Lambda_{\alpha}}\sim\|g\|_{\Lambda_{\alpha,{\mathcal D}}} +
A_{\alpha}(g)\,. \]
\end{theorema}
Furthermore, it is also true,  and the proof is given in
Proposition 2.1 below,  that
$\|g\|_{\Lambda_{\alpha}}\sim\|g\|_{\Lambda_{{\alpha},{\mathcal
D}_0}}$.  However, in this simpler formulation,  the tree
structure of the cubes in ${\mathcal D}$ has been lost.

 The proof of Theorem A relies on a close investigation of the predual
of $\Lambda_{\alpha}$,  namely, the Hardy space $H^p(R^N)$ with
$0<p= (\alpha +N)/N\le 1$. In the process  we  characterize $H^p$ in
terms of simpler subspaces: $H^{p}_{\mathcal D}$,  or dyadic $H^p$,
and $H^p_{{\mathcal S}_{\alpha}}$,   the space generated by the
special atoms in ${\mathcal S}_{\alpha}$.   Specifically, we  have
\begin{theoremb}Let $0<p\le 1$, and  $\alpha=N(1/p-1)$.
We then have
 \[H^p=H^p_{\mathcal D} +
H^p_{\mathcal S_{\alpha}},\] where the sum is understood in the
sense of quasinormed Banach spaces.
\end{theoremb}

The paper is organized as follows. In Section 1 we show that
individual $H^p$  atoms can be written as a superposition of dyadic
and special atoms; this fact may be thought of as an extension of
the one-dimensional result of Fridli concerning
$L^{\infty}\,1$-\,atoms, see \cite{fridli} and \cite{AST}.  Then, we
prove Theorem B. In Section 2 we discuss how to pass from
$\Lambda_{\alpha,{\mathcal D}}$, and $\Lambda_{\alpha,{\mathcal
D}_0}$, to the  Lipschitz space $\Lambda_{\alpha}$.
\section{Characterization of the Hardy spaces $H^p$}
We adopt the atomic definition of the Hardy spaces
 $H^p$, $0<p\le 1$, see \cite{GCRdF} and \cite {torchinsky}.
 Recall that a compactly supported function $a$
 with $[N(1/p-1)]$  vanishing
moments is an $L^{2}$ $p\,$-atom with defining cube $Q$  if
supp$(a)\subseteq Q$, and
\[|Q|^{1/p}\left(\frac1{|Q|} \int_Q|\,a(x)\,|^2 dx\right)^{1/2}\le 1\,.\]
 The Hardy space $H^p(R^N)=H^p$
 consists of those distributions $f$ that can be written
 as
$f=\sum \lambda_ja_j$, where the $a_j$'s are $H^p$ atoms, $\sum
|\lambda_j|^p<\infty$, and the convergence is in the sense of
distributions as well as in $H^p$.  Furthermore,
 \[ \|f\|_{H^p}\sim \inf\Big(\sum|\lambda_j|^p\Big)^{1/p}\,,\]
 where the  infimum is taken over all possible atomic decompositions
 of $f$. This last expression has traditionally been called the atomic $H^p$
 norm of $f$.

Collections of atoms with special properties can  be used to gain
a better understanding of the Hardy spaces. Formally, let
$\mathcal{A}$
 be a non-empty subset of $L^2$ $p\,$-atoms in the unit ball of $H^p$.
 The atomic space $H^p_{\mathcal A}$
 spanned by $\mathcal{A}$ consists of  those $\varphi$ in $H^p$
   of the form
\[\varphi=\sum \lambda_ja_j\,,\quad a_j\in
\mathcal{A}\,, \sum|\lambda_j|^p<\infty\,.\] It is readily seen
that, endowed with the atomic norm
\[\|\varphi\|_{H^p_{\mathcal A}}=\inf\Big\{\Big(\sum|\lambda_j|^p\Big)^{1/p}:
\varphi=\sum\lambda_j\,a_{j} \,,a_{j}\in {\mathcal A}\,\Big\}\,,\]
$H^p_{\mathcal A}$ becomes a complete quasinormed space. Clearly,
$H^p_{{\mathcal A}}\subseteq H^p$, and, for $f\in H^p_{{\mathcal
 A}}$, $\|f\|_{H^p} \le \|f\|_{H^p_{\mathcal A}}$.

 Two families are of particular interest to us. When ${\mathcal
 A}$ is the collection of all $L^2$ $p\,$-atoms whose defining cube is
 dyadic,
  the resulting space is $H^p_{\mathcal D}$, or dyadic $H^p$.
 Now, although $\|f\|_{H^p}\le  \|f\|_{H^p_{\mathcal D}}$, the two
quasinorms are not equivalent on $H^p_{\mathcal D}$. Indeed, for $p=
1$ and $N=1$, the functions
\[f_n(x)=2^{n} [\chi_{[1-2^{-n},1]}(x)-\chi_{[1,1+2^{-n}]}(x)]\,,\]
satisfy  $\|f_n\|_{H^1}= 1$, but $\|f_n\|_{H^1_{\mathcal D}}\sim
|n|$ tends to infinity with $n$.

Next, when ${\mathcal S}_{\alpha}$ is the family of piecewise
polynomial splines constructed above
 with $\alpha=N(1/p-1)$, in analogy with
the one-dimensional results in \cite {desouza} and \cite{AST},
$H^p_{{\mathcal S}_{\alpha}}$ is referred to as the space
generated by special atoms.

We are now ready to  describe $H^p$  atoms as a superposition of
dyadic and special atoms.
\begin{lemma} Let $a$ be  an $L^2$ $p\,$-atom with defining cube
$Q$,  $0<p\le 1$,  and  $\alpha=N(1/p-1)$. Then
 $a$ can be written as a linear combination of
 $2^N$ dyadic
atoms $a_{i}$, each supported in one of the dyadic subcubes of the
smallest special cube $Q_{n,k}$ containing $Q$, and a special atom
$b$ in ${\mathcal S}_{\alpha}$. More precisely,
$a(x)=\sum_{i=1}^{2^N}d_i\, a_{i}(x)+\sum_{L=1}^M
c_L\,p^L_{-n,-k,\alpha}(x)$, with $|d_i|\,,\,|c_L|\le c$.
 \end{lemma}
\begin{proof}
Suppose first that the defining cube of $a$ is  $Q_0$, and let
$Q_1,\ldots,Q_{2^N}$ denote the dyadic subcubes of $Q_0$.
Furthermore, let $\{e^1_i,\ldots, e^{M}_i\}$ denote an orthonormal
basis of   the subspace $A_i$ of $L^2(Q_i)$ consisting of
polynomials in ${\mathcal P}_{[\alpha]}$, $1\le i\le 2^N$. Put
\[\alpha_i(x)=a(x)\chi_{Q_i}(x)-\sum_{j=1}^M\langle a\chi_{Q_i},
e_j^i\rangle\,e_j^i(x)\,,\quad 1\le i\le 2^N\,,\] and observe that
$\langle \alpha_i, e_j^i\rangle =0$ for $1\le j\le M$. Therefore,
$\alpha_i$ has $[\alpha]$ vanishing moments, is supported in $Q_i$,
and
\[\|\alpha_i\|_2\le \|a\chi_{Q_i}\|_2+\sum_{j=1}^M\|a\chi_{Q_i}\|_2\le
(M+1)\,\|a\chi_{Q_i}\|_2\,.
\]
So, \[ a_i(x)=\frac{2^{N(1/2\,-1/p)}}{M+1}\,\alpha_i(x)\,,\quad 1\le
i\le N\,,\]
 is an $L^2$
$p\,$-\,dyadic atom. Finally, put
\[b(x)=a(x)-\frac{M+1}{2^{N(1/2\,-1/p)}}\sum_{i=1}^{2^N}a_i(x)\,.\]
Clearly $b$ has  $[\alpha]$ vanishing moments, is supported in
$Q_0$, coincides with a polynomial in  ${\mathcal P}_{[\alpha]}$ on
each dyadic subcube of $Q_0$, and
\[ \|b\|^2_2\le \sum_{i=1}^{2^N}\sum_{j=1}^M |\langle a\chi_{Q_i},
e_j^i\rangle|^2\le M\,\|a\|_2^2\,.\] So, $b\in A$, and,
consequently, $b(x)=\sum_{L=1}^Mc_L\,p^L(x)$, where
\[|c_L|=|\langle b,
p^L\rangle|\le c\,,\quad 1\le L\le M\,.\]

In the general case, let  $Q$ be the defining cube of $a$,
side-length $Q=\ell$,
  and let $n$ and  $k=(k_1,\dots,k_N)$ be chosen so that
$2^{n-1}\le \ell <2^n$, and
\[Q\subset [(k_1-1)2^n,
(k_1+1)2^n]\times \cdots \times[(k_N-1)2^n, (k_N+1)2^n]\,.\] Then,
$(1/2)^N\le |Q|/2^{nN}<1$.

Now, given $x\in Q_0$, let $a'$ be the translation and dilation of
$a$ given by
\[a'(x)=2^{nN/p}a(2^{n}x_1-k_1,\ldots,2^{n}x_N-k_N)\,.
\]
Clearly, $[\alpha]$ moments of $a'$ vanish, and
\[\|a'\|_2=2^{nN/p}\,2^{-nN/2}\|a\|_2\le c\,|Q|^{1/p}|Q|^{-1/2}\|a\|_2
\le c\,.
\]
Thus, $a'$ is a multiple of an atom with defining cube $Q_0$. By
the first part of the proof,
\[a'(x)=\sum_{i=1}^{2^N}d_i\,a_i'(x)+\sum_{L=1}^M c_L\,p^L(x)\,,\quad x\in
Q_0\,.
\]
The support of each $a_i'$ is contained in one of the dyadic
subcubes of $Q_0$, and, consequently, there is a $k$ such that
\[a_i(x)=2^{-nN/p}a_i'(2^{-n}x_1-k_1,\ldots,2^{-n}x_N-k_N)
\]
$a_i$ is an $L^2 p$ -atom supported in one of the dyadic subcubes of
$Q$. Similarly for the $p_L$'s. Thus,
\[a(x)=\sum_i d_i\,a_i(x)+\sum_{L=1}^M c_L p^L_{-n,-k,N(1/p-1)}(x)\,,\]
and we have finished.
 \taf
\end{proof}
Theorem B follows  readily from  Lemma 1.1. Clearly, $H^p_{\mathcal
D}+H^p_{{\mathcal S}_{\alpha}} \hookrightarrow H^p$. Conversely, let
$f=\sum_j\lambda_j\,a_j$ be in $H^p$. By Lemma 1.1  each $a_j$ can
be written as a sum of dyadic and special atoms, and,  by
distributing the sum, we can write $f=f_d+f_s$, with $f_d$ in
$H^p_{\mathcal D}$, $f_s$ in $H^p_{{\mathcal S}_{\alpha}}$, and
\[\|f_d\|_{ H^p_{\mathcal D}}, \|f_s\|_{ H^p_{{\mathcal
S}_{\alpha}}} \le c\, \Big(\sum |\lambda_j|^p\Big)^{1/p}\,.\] Taking
the infimum over the decompositions of $f$ we get $
\|f\|_{H^p_{\mathcal D}+H^p_{{\mathcal S}_{\alpha}}}\le c\,
\|f\|_{H^p}$, and $H^p\hookrightarrow H^p_{\mathcal
D}+H^p_{{\mathcal S}_{\alpha}}$. This completes the proof.

The meaning of this decomposition is the following. Cubes in
${\mathcal D}$ are contained in one of the   $2^N$ non-overlapping
quadrants of $R^N$. To allow for the information carried by a dyadic
cube to be transmitted to an adjacent dyadic cube, they must be
connected. The $p^L_{n,k,\alpha}$'s channel information across
adjacent dyadic cubes which would otherwise remain disconnected. The
reader will have no difficulty in proving the quantitative version
of this observation: Let $T$ be a linear mapping defined on $H^p$,
$0<p\le 1$, that assumes values in a quasinormed Banach space $X$.
Then, $T$ is continuous if, and only if, the restrictions of $T$ to
$H^p_{\mathcal D}$ and $H^p_{{\mathcal S}_{\alpha}}$ are continuous.
\section{Characterizations of $\Lambda_{\alpha}$}
Theorem A  describes how to pass from $\Lambda_{\alpha,{\mathcal
D}}$ to $\Lambda_{\alpha}$, and we prove it next. Since
$(H^p)^*=\Lambda_{\alpha}$ and $(H^p_{\mathcal
D})^*=\Lambda_{\alpha, {\mathcal D}}$, from Theorem B it follows
readily that  $\Lambda_{\alpha}= \Lambda_{\alpha, {\mathcal D}}\cap
(H^p_{{\mathcal S}_{\alpha}})^*$, so it only remains to show that
$(H^p_{{\mathcal S}_{\alpha}})^*$ is characterized by the condition
$A_{\alpha}(g)<\infty$.

First note that if $g$ is a locally square-integrable function
with $A_{\alpha}(g)<\infty$ and
$f=\sum_{j,L}c_{j,L}\,p^L_{n_j,k_j,\alpha}$, since $0<p\le 1$,
\begin{align*}|\langle g,f\rangle| &\le \sum_{j,L}|c_{j,L}|
~|\langle g, p^L_{n_j,k_j,\alpha}\rangle|\\
&\le A_{\alpha}(g)\bigg[\sum_{j,L}|c_{j,L}|^p\bigg]^{1/p},
\end{align*}
and, consequently, taking the infimum over all atomic decompositions
of $f$ in ${H^p_{{\mathcal S}_{\alpha}}}$,  we get $g\in
(H^p_{{\mathcal S}_{\alpha}})^*$ and $\|g\|_{(H^p_{{\mathcal
S}_{\alpha}})^*}\le A_{\alpha}(g)$.

To prove the converse we proceed  as in \cite{APCAT}.  Let
$Q_n=[-2^n,2^n]^N$. We begin by observing that functions $f$ in
$L^2(Q_n)$ that have vanishing moments up to order $[\alpha]$ and
coincide with polynomials of degree $[\alpha]$ on the dyadic
subcubes of $Q_n$ belong to $H^p_{{\mathcal S}_{\alpha}}$ and
\[\|f\|_{H^p_{{\mathcal S}_{\alpha}}}\le |Q_n|^{1/p-1/2}\|f\|_2\,.\]
 Given $\ell\in (H^p_{{\mathcal S}_{\alpha}})^*$, for a fixed $n$ let us consider the
restriction of $\ell$ to the space of $L^2$ functions $f$ with
$[\alpha]$ vanishing moments that are supported in $Q_n$. Since
\[
|\ell(f)|\le \|\ell\| \,\|f\|_{H^p_{{\mathcal S}_{\alpha}}}\le
\|\ell\|\,|Q_n|^{1/p-1/2}\|f\|_2\,,
\]
 this restriction is
continuous with respect to the norm in $L^2$ and, consequently, it
can be extended to a continuous linear functional in $L^2$ and
represented as
\[\ell(f)=\int_{Q_n}f(x)\,g_n(x)\,dx\,,\]
where $ g_n\in L^2(Q_n)$ and satisfies $\|g_n\|_2\le
\|\ell\|\,|Q_n|^{1/p-1/2}$. Clearly, $g_n$ is uniquely determined in
$Q_n$ up to a polynomial $p_n$ in ${\mathcal P}_{[\alpha]}$.
Therefore,
\[g_n(x)-p_n(x) =g_m(x)-p_m(x)\,,\quad {\rm a.e. }
\ x\in Q_{\min(n,m)}\,.\]
Consequently, if
\[g(x)=g_n(x)-p_n(x)\,,\quad  x\in Q_n\,,\]
$g(x)$ is well defined a.e. and, if $f\in L^2$
 has $[\alpha]$  vanishing moments and is supported in $Q_n$,
  we have
 \begin{align*}
 \ell(f)& =\int_{R^N} f(x)\,g_n(x)\,dx\\
 &= \int_{R^N} f(x)\,
 [g_n(x)-p_n(x)]\,dx\\
 &=\int_{R^N} f(x)\,g(x)\,dx\,.
 \end{align*}
Moreover, since each  $2^{nN/p} p^L(2^n\cdot+k)$ is an $L^2$
$p$-atom, $1\le L\le M$,  it readily follows that
\begin{align*}
A_{\alpha}(g) &=\sup_{1\le L\le M}\sup_{n,k\in Z}|\langle g,
2^{-n/p}
p^L(2^n\cdot+k)\rangle|\\
&\le \|\ell\|\, \sup_L \|p^L\|_{H^p} \le \,\|\ell\|\,,
\end{align*}
and, consequently, $A_{\alpha}(g)\le\|\ell\|\,,$ and
$(H^p_{{\mathcal S}_{\alpha}})^*$ is the desired space. {\hskip 5pt}
$\blacksquare$

The reader will have no difficulty in showing that this result
implies the following: Let $T$ be a bounded linear operator from a
quasinormed space $X$ into $\Lambda_{\alpha,{\mathcal D}}$. Then,
$T$ is bounded from $X$ into   $\Lambda_{\alpha}$ if, and only if,
$A_{\alpha}(Tx)\le c\,\|x\|_X$ for every $x\in X$.

The process of averaging the translates of dyadic  BMO functions
leads to BMO,   and is an important tool in obtaining results in BMO
once they are known to be true in its dyadic counterpart, ${\rm
BMO}_d$, see \cite{garnettjones}. It is also known that BMO can be
obtained as the intersection of ${\rm BMO}_d$ and one of its shifted
counterparts,  see \cite{tm}. These results motivate our next
proposition, which essentially says that $g\in \Lambda_{\alpha}$ if,
and only if, $g\in \Lambda_{\alpha,{\mathcal D}}$ and $g$ is in the
Lipschitz class obtained from the shifted dyadic grid. Note that the
shifts involved in this class are in all directions parallel to the
coordinate axis and  depend on the side-length of the cube.
\begin{proposition}
${\Lambda_{\alpha}}=\Lambda_{\alpha, {\mathcal D}_0}$, and
$\|g\|_{\Lambda_{\alpha}}\sim \|g\|_{\Lambda_{\alpha, {\mathcal
D}_0}}$.
\end{proposition}
\begin{proof} It is obvious that
$ \|g\|_{\Lambda_{\alpha,{\mathcal D}_0}}\le
\|g\|_{\Lambda_{\alpha}}$. To show the other inequality we invoke
 Theorem A.  Since $ {\mathcal D}\subset {\mathcal
D}_0$, it suffices to estimate $A_{\alpha}(g)$, or, equivalently,
$|\langle g, p \rangle|$ for  $p\in {\mathcal S}_{\alpha}$,
$\alpha=N(1/p-1)$. So, pick $p=p^L_{n,k,\alpha}$ in ${\mathcal
S}_{\alpha}$. The defining cube $Q$ of $p^L_{n,k,\alpha}$ is in
${\mathcal D}_0$, and, since $p^L_{n,k,\alpha}$ has $[\alpha]$
vanishing moments, $\langle p^L_{n,k,\alpha},p_Q(g)\rangle=0$.
Therefore,
\begin{align*}
|\langle g, p^L_{n,k,\alpha} \rangle| &= |\langle g -p_Q(g),
p^L_{n,k,\alpha} \rangle|\\
&\le \|p^L_{n,k,\alpha}\|_2\,\|g -p_Q(g)\|_{L^2(Q)}\\
&\le |Q|^{\alpha/N}|Q|^{1/2}\|p^L_{n,k,\alpha} \|_2\,
\|g\|_{\Lambda_{\alpha,{\mathcal D}_0}}.
\end{align*}
Now,  a simple change of variables gives
$|Q|^{\alpha/N}|Q|^{1/2}\|p^L_{n,k,\alpha} \|_2\le 1$, and,
consequently,  also $A_{\alpha}(g)\le
\|g\|_{\Lambda_{\alpha,{\mathcal D}_0}}$.
 \taf
 \end{proof}

\end{document}